\pdfoutput=1
\RequirePackage{ifpdf}
\ifpdf 
\documentclass[pdftex]{sigma}
\else
\documentclass{sigma}
\fi

\usepackage{mathrsfs,amscd}

\input xy \xyoption{all}

\DeclareMathOperator\id{id}
\DeclareMathOperator\Sym{Sym}
\DeclareMathOperator\Hom{Hom}
\DeclareMathOperator\pt{pt}
\DeclareMathOperator\Stab{Stab}
\DeclareMathOperator\Gr{Gr}
\newcommand{\bK}{\mathsf{K}}
\newcommand{\bG}{\mathsf{G}}
\def\Z{{\mathbb Z}}
\def\P{{\mathbb P}}
\def\R{{\mathbb R}}
\def\C{{\mathbb C}}
\newcommand{\bT}{\mathsf{T}}
\newcommand{\Or}{\textsf{O}}
\def\Oc{{\mathcal O}}
\newcommand{\sE}{\mathscr{E}}
\newcommand{\cE}{\mathscr{E}}
\DeclareMathOperator\Pic{Pic}
\def\ttt{{\boldsymbol t}}
\def\zz{{\boldsymbol z}}
\def\mub{{\boldsymbol\mu}}
\newcommand{\Spec}{\operatorname{Spec}}
\newcommand{\Ell}{\operatorname{Ell}}
\newcommand{\cV}{\mathcal{V}}
\let\thi\vartheta
\let\pho\phi

\numberwithin{equation}{section}
\newtheorem{Theorem}{Theorem}[section]
\newtheorem{Corollary}[Theorem]{Corollary}
\newtheorem{Lemma}[Theorem]{Lemma}
\newtheorem{Proposition}[Theorem]{Proposition}
 { \theoremstyle{definition}
\newtheorem{Definition}[Theorem]{Definition}
\newtheorem{Note}[Theorem]{Note}
\newtheorem{Example}[Theorem]{Example}
\newtheorem{Remark}[Theorem]{Remark} }

\begin{document}

\allowdisplaybreaks

\newcommand{\arXivNumber}{1906.00134}

\renewcommand{\thefootnote}{}

\renewcommand{\PaperNumber}{093}

\FirstPageHeading

\ShortArticleName{Mirror Self-Symmetry of the Cotangent Bundle of the Full Flag Variety}

\ArticleName{Three-Dimensional Mirror Self-Symmetry \\ of the Cotangent Bundle of the Full Flag Variety\footnote{This paper is a~contribution to the Special Issue on Representation Theory and Integrable Systems in honor of Vitaly Tarasov on the 60th birthday and Alexander Varchenko on the 70th birthday. The full collection is available at \href{https://www.emis.de/journals/SIGMA/Tarasov-Varchenko.html}{https://www.emis.de/journals/SIGMA/Tarasov-Varchenko.html}}}

\Author{Rich\'ard RIM\'ANYI~$^{\dag^1\!}$, Andrey SMIRNOV~$^{\dag^1\dag^2\!}$, Alexander VARCHENKO~$^{\dag^1\dag^3\!}$ and Zijun ZHOU~$^{\dag^4}$}
\AuthorNameForHeading{R.~Rim\'anyi, A.~Smirnov, A.~Varchenko and Z.~Zhou}
\Address{$^{\dag^1}$~Department of Mathematics, University of North Carolina at Chapel Hill, \\
\hphantom{$^{\dag^1}$}~Chapel Hill, NC 27599-3250, USA}
\EmailDD{\href{mailto:rimanyi@email.unc.edu}{rimanyi@email.unc.edu}, \href{mailto:asmirnov@email.unc.edu}{asmirnov@email.unc.edu}, \href{mailto:anv@email.unc.edu}{anv@email.unc.edu}}
\Address{$^{\dag^2}$~Institute for Problems of Information Transmission, \\
\hphantom{$^{\dag^2}$}~Bolshoy Karetny 19, Moscow 127994, Russia}
\Address{$^{\dag^3}$~Faculty of Mathematics and Mechanics, Lomonosov Moscow State University, \\
\hphantom{$^{\dag^3}$}~Leninskiye Gory 1, 119991 Moscow GSP-1, Russia}
\Address{$^{\dag^4}$~Department of Mathematics, Stanford University, \\
\hphantom{$^{\dag^4}$}~450 Serra Mall, Stanford, CA 94305, USA}
\EmailDD{\href{mailto:zz2224@stanford.edu}{zz2224@stanford.edu}}

\ArticleDates{Received July 08, 2019, in final form November 18, 2019; Published online November 28, 2019}

\Abstract{Let $X$ be a holomorphic symplectic variety with a torus $\mathsf{T}$ action and a finite fixed point set of cardinality $k$. We assume that elliptic stable envelope exists for $X$. Let $A_{I,J}= \operatorname{Stab}(J)|_{I}$ be the $k\times k$ matrix of restrictions of the elliptic stable envelopes of~$X$ to the fixed points. The entries of this matrix are theta-functions of two groups of variables: the K\"ahler parameters and equivariant parameters of~$X$. We say that two such varieties~$X$ and~$X'$ are related by the 3d mirror symmetry if the fixed point sets of~$X$ and~$X'$ have the same cardinality and can be identified so that the restriction matrix of~$X$ becomes equal to the restriction matrix of~$X'$ after transposition and interchanging the equivariant and K\"ahler parameters of~$X$, respectively, with the K\"ahler and equivariant parameters of~$X'$. The first examples of pairs of 3d symmetric varieties were constructed in [Rim\'anyi~R., Smirnov~A., Varchenko~A., Zhou~Z., arXiv:1902.03677], where the cotangent bundle $T^*\operatorname{Gr}(k,n)$ to a~Grassmannian is proved to be a 3d mirror to a Nakajima quiver variety of $A_{n-1}$-type. In this paper we prove that the cotangent bundle of the full flag variety is 3d mirror self-symmetric. That statement in particular leads to nontrivial theta-function identities.}


\Classification{17B37; 55N34; 32C35; 55R40}

\renewcommand{\thefootnote}{\arabic{footnote}}
\setcounter{footnote}{0}

\section{Introduction}

\subsection{The 3d mirror symmetry}

The 3d mirror symmetry has recently received plenty of attention in both representation theory and mathematical physics. It was introduced by various groups of physicists in \cite{BDG, BDGH, PhysMir2, PhysMir1, Ga-Wit, HW, PhysMir3}, where one starts with a pair of 3d $\mathcal{N} = 4$ supersymmetric gauge theories, considered as mirror to each other. Under the mirror symmetry, the two interesting components -- \emph{Higgs branch} and \emph{Coulomb branch} -- of the moduli spaces of vacua are interchanged, as well as the Fayet--Iliopoulos parameters and mass parameters.

Translated into the mathematical language, the $\mathcal{N} = 4$ supersymmetry indicates a~hyperk\"ahler structure on the moduli space. In particular, for the theories we are interested in, the Higgs branch $X$ is a variety which can be constructed as a hyperk\"ahler quotient, or equivalently in the algebraic setting, as a holomorphic symplectic quotient. As a large class of examples, Nakajima quiver varieties arise in this way, as Higgs branches of $\mathcal{N}=4$ supersymmetric quiver gauge theories. The mass parameters arise here as equivariant parameters of a certain torus $\bT$ acting naturally on the Higgs branch $X$. The Fayet--Iliopoulos parameters, or K\"ahler parameters arise as coordinates on the torus $\bK=\mathrm{Pic}(X)\otimes_{\Z} \C^{\times}$.

The ``dual'' symplectic varieties $X'$ -- Coulomb branches, however, did not admit a mathematical construction until recently,
see \cite{BFN2, NakCoulomb1, Coulomb1}, where the Coulomb branches are defined as singular affine schemes by taking spectrums of certain convolution algebras, and quantized by considering noncommutative structures. Nevertheless, in many special cases, Coulomb branches admit nice resolutions, and can be identified with the Higgs branches of the mirror theory. These cases include hypertoric varieties, cotangent bundles of partial flag varieties, the Hilbert scheme of points on $\C^2$ and more generally, moduli spaces of instantons on the minimal resolution of~$A_n$ singularities. 3d mirror symmetry is often referred to as \emph{symplectic duality} in mathematics, see references in~\cite{BLPW, matdu}.

Aganagic and Okounkov in \cite{AOelliptic} argue that the equivariant elliptic cohomology and the theory of elliptic stable envelopes provide a natural framework to study the 3d mirror symmetry (See also the very important talk ``Enumerative symplectic duality'' given by A.~Okounkov during the 2018 MSRI workshop ``Structures in Enumerative Geometry'').
In particular, they argue that the elliptic stable envelopes of a symplectic variety depend on both equivariant and K\"ahler parameters in a symmetric way. Motivated by~\cite{AOelliptic} we give the following definition of 3d mirror symmetric pairs of symplectic varieties~$X$ and~$X'$.

Let a symplectic variety $X$ be endowed with a Hamiltonian action of a torus $\bT$. Let the set~$X^{\bT}$ of torus fixed points be a~finite set of cordiality~$k$. For $I\in X^{\bT}$ let $\Stab(I)$ be the elliptic stable envelope of~$I$.\footnote{For the generality in which elliptic stable envelope can be defined see \cite[Chapter~3]{MO}. The existence of these classes is proven for $X$ given by Nakajima varieties and hypertoric varieties. It is expected, however, that elliptic stable envelopes exist for more general symplectic varieties.} It is a class in elliptic cohomology of~$X$. The restrictions of these elliptic cohomology classes to points of~$X^{\bT}$ give a $k\times k$ matrix $A_{I,J}= \Stab(I) |_{J}$. The matrix elements~$A_{I,J}$ are theta functions of two sets of variables associated with~$X$: the equivariant parameters, which are coordinates on the torus~$\bT$, and the K\"ahler parameters, which are coordinates on the torus $\bK=\mathrm{Pic}(X)\otimes_{\Z} \C^{\times}$.

Let $X$ and $X'$ be two such symplectic varieties.

\begin{Definition} \label{maindefin}
 A variety $X'$ is a 3d mirror of a variety $X$ if
 \begin{enumerate}\itemsep=0pt
 \item[(1)] There exists a bijection of fixed point sets $X^{\bT}\rightarrow (X')^{\bT'}$, $I\mapsto I'$.
 \item[(2)] There exists an isomorphism
 \begin{gather*}
 \kappa\colon \ \bT \rightarrow \bK', \qquad \bK \rightarrow \bT'
 \end{gather*}
 identifying the equivariant and K\"ahler parameters of $X$ with, respectively, K\"ahler and equivariant parameters of $X'$.
 \item[(3)] The matrices of restrictions of elliptic stable envelopes for $X$ and $X'$ coincide after transposition (when the set of fixed points are identified by (1)) and change of variables~(2):
 \begin{gather} \label{acoin}
 A_{I,J}=\kappa^{*}( A'_{J',I'}),
 \end{gather}
 where $A'_{J',I'}$ denotes the restriction matrix of elliptic stable envelopes for $X'$.
 \end{enumerate}
\end{Definition}

The first examples of pairs of 3d symmetric varieties were constructed in~\cite{RVSZ}, where the cotangent bundle $T^*\Gr(k,n)$ to a Grassmannian is proved to be
a 3d mirror of a Nakajima quiver variety of $A_{n-1}$-type. In this paper we prove that the cotangent bundle of the full flag variety is 3d mirror self-symmetric.

That statement in particular leads to nontrivial theta-function identities. The left and right-hand sides of equation~(\ref{acoin}) are given as sums of alternating products of Jacobi theta functions in two groups of variables. Equality~(\ref{acoin}) provides $k^2$ highly nontrivial identities satisfied by Jacobi theta functions. In Section~\ref{exsection} we describe some of these identities in detail.

Alternatively, one could define 3d mirror variety $X'$ as a variety which has the same $K$-theoretic vertex functions (after the corresponding change of the equivariant and K\"ahler para\-me\-ters). The vertex functions of $X$ are the $K$-theoretic analogues of the Givental's J-functions introduced in~\cite{pcmilec}. For the cotangent bundles of full flag varieties the vertex functions were studied for example in \cite{Korot1,Kor2,Push2}. We believe that this alternative definition is equivalent to the one we give above.

\subsection{Elliptic stable envelopes: main results}

The notion of stable envelopes is introduced by Maulik--Okounkov in \cite{MO} to study the quantum cohomology of Nakajima quiver varieties. Stable envelopes depend on a choice of a cocharacter of the torus $\bT$. The Lie algebra of the torus admits a wall-and-chamber structure, such that the transition matrices between stable envelopes for different chambers turn out to be certain $R$-matrices satisfying the Yang--Baxter equations, and hence they define quantum group structures. In~\cite{pcmilec, OS}, the construction is generalized to K-theory, realizing the representations of quantum affine algebras. What appears new in K-theoretic stable envelopes is the piecewise linear dependence on a choice of \emph{slope}, which lives in the space of K\"ahler parameters.

The slope dependence is replaced by the meromorphic dependence on a complex K\"ahler parameters $\mu \in \bK$ (in the original paper~\cite{AOelliptic} the K\"ahler parameters are denoted by $z$), in the further generalization of stable envelopes to equivariant elliptic cohomology, from which the cohomological and K-theoretic analogs can be obtained as certain limits. Now the elliptic stable envelopes depend on both equivariant and K\"ahler parameters, which makes the 3d mirror symmetry phenomenon possible.

In this paper, we will consider the special case where $X$ is the cotangent bundle of the variety of complete flags in $\C^n$, which can be constructed as the Nakajima quiver variety associated to the $A_{n-1}$-quiver with dimension vector $(1, 2, \dots, n-2, n-1)$ and framing vector $(0, 0, \dots, 0, n)$. There is a torus action induced by the torus~$\bT$ on the framing space $\C^n$. Fixed points $X^\bT$ can be identified with permutations of the ordered set $(1,2, \dots, n)$, and hence parameterized by the symmetric group~$\mathfrak{S}_n$.

Let $q\in \C^*$ be a complex number with $|q|<1$, and $E = \C^* / q^{\Z}$ be the elliptic curve with modular parameter $q$. By definition, the extended equivariant elliptic cohomology $\textsf{E}_\bT(X)$ of $X$ fits into the following diagram
\begin{gather} \label{incdia}\begin{split}&
\xymatrix{
 \widehat \Or_I \ar@{^{(}->}[r] & \textsf{E}_\bT(X) \ar[d] \ar@{^{(}->}[r] & \mathsf{S}(X) \times \sE_\bT \times \sE_{\Pic(X)} \\
 & \sE_\bT \times \sE_{\Pic(X)} ,
}\end{split}
\end{gather}
where $\mathsf{S}(X)=\prod\limits_{k=1}^{n-1} S^k E$ is the space of Chern roots, $\sE_\bT$ and $\sE_{\Pic (X)}$ are the spaces of equivariant and K\"ahler parameters respectively, and $\widehat\Or_I$ is an irreducible component of $\textsf{E}_\bT(X)$, associated with the fixed point $I$, called an orbit appearing in the following decomposition given by the localization
\begin{gather*}
\textsf{E}_{\bT}(X)=\bigg(\coprod\limits_{I\in X^{\bT}} \widehat\Or_{I}\bigg) /\Delta.
\end{gather*}
Here each $\widehat\Or_I$ is isomorphic to the base $\sE_\bT \times \sE_{\Pic (X)}$, and $\Delta$ denotes the gluing data.

Moreover, in our case $X$ is a \emph{GKM variety}, which by definition means that it admits finitely many $\bT$-fixed points and finitely many 1-dimensional orbits, and implies that $\textsf{E}_\bT(X)$ above is a~simple normal crossing union of the orbits $\widehat\Or_I$, along hyperplanes that can be explicitly described.

The dual variety of $X$ is another copy of the cotangent bundle of complete flag variety, which we denote by $X'$, in order to distinguish it from $X$. From the perspective of the 3d mirror symmetry, although $X$ and $X'$ are isomorphic as varieties, we do not identify them in this naive way. Instead, we consider the sets of fixed points of $X$ and $X'$ which are both parameterized by permutations $I \in \mathfrak{S}_n$, and define a natural bijection between the fixed points as
\begin{gather*}
{\sf bj}\colon \ X^\bT \xrightarrow{\sim} (X')^{\bT'}, \qquad I \mapsto I^{-1},
\end{gather*}
where $I^{-1}$ denotes the permutation inverse to $I$. Moreover, we also identify the base spaces of parameters in a nontrivial way
\begin{gather} \label{kap}
\kappa\colon \ \cE_{{\rm Pic}(X)}\cong \cE_{\bT^{'}}, \qquad \cE_{{\rm Pic}(X^{'})}\cong \cE_{\bT},
\\
\hphantom{\kappa\colon}{} \ \mu'_i \mapsto z_i, \qquad z'_i \mapsto \mu_i, \qquad \hbar' \mapsto \hbar.\nonumber
\end{gather}

By definition, given a fixed point $I\in X^\bT$, and a chosen cocharacter $\sigma$ of $\bT$, the elliptic stable envelope $\Stab_\sigma (I)$ is the section of a certain line bundle $\mathcal{T} (I)$ on $\textsf{E}_\bT (X)$, uniquely determined by a set of axioms. Moreover, explicit formulas for this sections, in terms of theta functions, can be obtained via abelianization. We will be interested in their restrictions to orbits $\Stab (I)|_{\hat{\Or}_{J}}$, and the normalized version ${\bf Stab}(I)|_{\hat{\Or}_{J}}$.

Our main result will be the following identity of the normalized restriction matrices of elliptic stable envelopes, for $X$ and $X'$.

\begin{Theorem} \label{introthm} Let $I,J \in X^{\bT}$ be fixed points and $I^{-1}$, $J^{-1}$ be the corresponding fixed points on the dual variety. Then
 \begin{gather} \label{relenv}
 {\bf Stab}(I)|_{\widehat{\Or}_{J}}=\kappa^{*}\big( {\bf Stab}'\big(J^{-1}\big)\big|_{\widehat{\Or}'_{I^{-1}}}\big).
 \end{gather}
\end{Theorem}
Here $\kappa\colon \widehat{\Or}_{J} \rightarrow \widehat{\Or}'_{I^{-1}}$ is the isomorphism (\ref{kap}) and the equality (\ref{relenv}) means that the corresponding sections coincide after this change of variables.

Moreover, by the Fourier--Mukai philosophy, a natural idea originally from Aganagic--Okoun\-kov~\cite{AOelliptic} is to enhance the coincidence above to the existence of a universal \emph{duality interface}\footnote{In the previous paper \cite{RVSZ}, it is called the \emph{Mother function}. } on the product $X \times X'$. Consider the following diagram of embeddings
\begin{gather*}
X\times \{ J \} \stackrel{i_{J}}{\longrightarrow} X\times X' \stackrel{i_{I}}{\longleftarrow} \{I\}\times X'.
\end{gather*}
Theorem \ref{introthm} can then be rephrased as
\begin{Theorem}
 There exists a holomorphic section ${\mathfrak{m}}$ $($the duality interface$)$ of a certain line bundle on ${\rm Ell}_{\bT\times\bT'}(X\times X')$ such that
 \begin{gather*}
 (i^{*}_{J})^*({\mathfrak{m}})= {\bf Stab}(I), \qquad (i^{*}_{I})^*({\mathfrak{m}})= {\bf Stab}' (J),
 \end{gather*}
 where $I$ is a fixed point on $X$ and $J$ is the corresponding fixed point on $X'$ $($i.e., $J=I^{-1}$ as a~permutation$)$.
\end{Theorem}

\subsection[Weight functions and $R$-matrices]{Weight functions and $\boldsymbol{R}$-matrices}

Our proof of Theorem \ref{relenv} relies on the observation that the elliptic stable envelope $\Stab_\sigma (I)$, as defined in Aganagic--Okounkov \cite{AOelliptic}, is related to \emph{weight functions} $W^{\sigma}_I (\ttt, \zz, \hbar, \mub)$, defined in~\cite{RTV}. The weight function $W^{\sigma}_I (\ttt, \zz, \hbar, \mub)$ is a section of a certain line bundle over $\mathsf{S}(X) \times \sE_\bT \times \sE_{\Pic(X)}$ in (\ref{incdia}). The elliptic stable envelope $\Stab_\sigma (I)$ is the restriction of this section to
the extended elliptic cohomology $\textsf{E}_\bT(X)$.

Weight functions first arise as integrands in the integral presentations of solutions to qKZ equations, associated with certain Yangians of type A \cite{FTV, FTV2,TV4, TV5, TV6, TV7, Var}. For us, the weight functions here are the elliptic version introduced in~\cite{RTV}.

Important properties of weight functions are described by the so called $R$-matrix relations. These relations describe the transformation properties of weight functions under the permutations of equivariant parameters. We show that these relations, in fact, uniquely determine the restriction matrices~$A_{I,J}$.

Similar relations, describing the transformations of weight functions under the permutations of K\"ahler parameters were recently found by Rim\'anyi--Weber in~\cite{RW}. The proof of our main theorem is based on the observation that these new relations can be understood as the $R$-matrix relations for the 3d mirror variety $X'$ (because the K\"ahler parameters of $X$ is identified with equivariant parameters of $X'$ under the 3d mirror symmetry). The $R$-matrix relations and the dual $R$-matrix relations then provide two ways to compute the restriction matrices, which is essentially two sides of the main equality of Theorem~\ref{introthm}.

Let us note that 3d self symmetry of full flag varieties should have important applications to representation theory. In particular, we expect that it is closely related to self-symmetry of double affine Hecke algebra under the Cherednik's Fourier transform~\cite{Cher1}. Another interesting example of a symplectic variety which is 3d mirror self-dual is the Hilbert scheme of points on the complex plane ${\rm Hilb}^{n}\big(\C^2\big)$. The explicit formulas for the elliptic stable envelopes in this case were obtained in~\cite{EllHilb}. In this case, however, ${\rm Hilb}^{n}\big(\C^2\big)$ is not a GKM variety and therefore methods used in this paper are unavailable.

We remark also that this paper deals with the cotangent bundles of full flag varieties of $A$-type. In general, it is natural to expect that cotangent bundles of the full flag varieties for a group $G$ is a~3d mirror of the cotangent bundle of full flag variety for the Langlands dual group~$^{L}G$. Though in general these flag varieties are not quiver varieties, both the $R$-matrix and the Bott--Samelson recursion~\cite{RW} is available in this setting and the 3d mirror symmetry can be proved using technique similar to one in the present paper.

\section[Equivariant elliptic cohomology of $X$]{Equivariant elliptic cohomology of $\boldsymbol{X}$}

In this section we give a brief introduction to equivariant elliptic cohomology. For detailed definitions and constructions, we refer the reader to \cite{ell1,ell2,ell3,ell4,ell5,ell6}, and also the recently appeared new approach \cite{BT}.

\subsection{The equivariant elliptic cohomology functor}

Let $X$ be a smooth quasiprojective variety over $\C$, and $\bT$ be a torus acting on $X$. Recall that $\bT$-equivariant cohomology is a contravariant functor from the category of varieties with $\bT$-actions to the category of algebras over the ring of equivariant parameters $H^*_\bT (\pt)$, which is naturally identified with affine schemes over $\Spec H^*_\bT (\pt) \cong \C^r$, where $r = \dim \bT$. Equivariant K-theory can be defined in a similar way, with the additive group $\C^r$ replaced by the multiplicative~$\Spec K_\bT (\pt) \cong (\C^{\times})^r$.

Let us set
\begin{gather*}
E := \C^{\times} / q^{\Z},
\end{gather*}
which is a family of elliptic curves parametrized by the punctured disk $0<|q|<1$. In the general definition of elliptic cohomology one works with more general families of elliptic curves, but considering $E$ will be sufficient for the purposes of the present paper.

Equivariant elliptic cohomology is constructed as a covariant functor
\begin{gather*}
\Ell_\bT\colon \ \{ \text{varieties with $\bT$-actions} \} \rightarrow \{ {\rm schemes}\},
\end{gather*}
for which the base space of equivariant parameters is
\begin{gather*}
\sE_\bT := \Ell_\bT (\pt) \cong E^r.
\end{gather*}
By functoriality, every $X$ with $\bT$-action is associated with a structure map $ \Ell_\bT(\pi)\colon \Ell_\bT (X) \to \Ell_\bT(\pt)$, induced by the projection $\pi\colon X \to \pt$.

We briefly describe the construction of equivariant elliptic cohomology. For each point $t\in \sE_\bT$, take a small analytic neighborhood $U_t$, which is isomorphic via the exponential map to a small analytic neighborhood in $\C^r$. Consider the sheaf of algebras
\begin{gather*}
\mathscr{H}_{U_t} := H^\bullet_\bT \big(X^{\bT_t}\big) \otimes_{H^\bullet_\bT (\pt)} \Oc^{{\rm an}}_{U_t},
\end{gather*}
where
\begin{gather*}
\bT_t := \bigcap_{\substack{\chi \in \mathrm{char} (\bT),\, \chi(\tilde t) \in q^{\Z} }} \ker \chi \subset \bT,
\end{gather*}
and $\tilde t \in \bT$ is any lift of $t \in \cE_\bT$.

Those algebras glue to a sheaf $\mathscr{H}$ over $\sE_\bT$, and we define $\Ell_\bT (X) := \Spec_{\sE_\bT} \mathscr{H}$. The fiber of $\Ell_\bT (X)$ over $t$ is obtained by setting local coordinates to $0$, as described in the following diagram~\cite{AOelliptic}:
\begin{gather*}
\xymatrix{
 \Spec H^{\bullet}\big(X^{\bT_{t}}\big) \ar@{^{(}->}[r] \ar[d]^{\Ell_\bT(\pi) } & \Spec H_\bT^{\bullet}\big(X^{\bT_{t}}\big) \ar[d] & (\pi^*)^{-1} (U_t) \ar[l] \ar[r] \ar[d] & \ar[d]^{\Ell_\bT(\pi) } {\rm Ell}_{\bT}(X) \\
 \{t\}
 \ar@{^{(}->}[r]^{} & \C^r & U_t \ar[l] \ar[r] &
 \sE_{\bT}. }
\end{gather*}
This diagram describes a structure of the scheme $\Ell_\bT (X)$ and gives one of several definitions of elliptic cohomology.

\subsection{Chern roots and extended elliptic cohomology}

In this subsection, we consider $X$ constructed as a GIT quotient of the form $Y /\!\!/\!_\theta \bG$, where $\bG$ is a linear reductive group acting on an affine space $\C^N$, $\theta$ is a fixed character of~$\bG$, and $Y \subset \C^N$ is a $\bG$-invariant subvariety. Let $\bT$ be a torus acting on $\C^N$ which commutes with $\bG$. The action hence descends to $X$.

Given a character $\chi\colon \bG \to \C^*$, the 1-dimensional $\bG$-representation $\C_\chi$ descends to a line bundle $L_\chi$ on the quotient $X$. In other words, consider the map
\begin{gather*}
X = Y^{ss} / \bG \subset [Y/\bG] \subset \big[\C^N / \bG\big] \to B\bG \xrightarrow{\chi} B\C^*.
\end{gather*}
The bundle $L_\chi$ is the pullback of the tautological line bundle on $B\C^*$ to $X$. More generally, any $\bG$-representation pulls back to a vector bundle, called a tautological bundle, on $X$.

Let $\bK\subset \bG$ be the maximal torus, and $W$ be the Weyl group. Then $\Ell_G (\pt) \cong E^{\dim \bK} / W$. From the diagram above, we have the cohomological Kirwan map
\begin{gather*}
H^*_\bK(\pt)^W \otimes H^*_\bT (\pt) \cong H^*_G (\pt) \otimes H^*_\bT (\pt) \to H^*_\bT(X),
\end{gather*}
and also the elliptic Kirwan map
\begin{gather} \label{Kirwan-ell}
\Ell_\bT(X) \to \big( E^{\dim \bK} / W \big) \times \sE_\bT.
\end{gather}

We say that $X$ satisfies \emph{Kirwan surjectivity}, if (\ref{Kirwan-ell}) is a closed embedding. By the results of~\cite{McGN}, it holds for any Nakajima quiver variety.

To include the dependence on K\"ahler parameters, consider
\begin{gather*}
\sE_{\Pic (X)} := \Pic (X) \otimes_{\Z} E \cong E^{\dim \Pic(X)},
\end{gather*}
and define the extended equivariant elliptic cohomology by
\begin{gather*}
\textsf{E}_\bT (X) := \Ell_\bT (X) \times \sE_{\Pic(X)}.
\end{gather*}
In particular, if $X$ is a GIT quotient satisfying Kirwan surjectivity, one has the embedding
\begin{gather*}
\xymatrix{
 \textsf{E}_\bT(X) \ar[d] \ar@{^{(}->}[r] & ( E^{\dim \bK} / W ) \times \sE_\bT \times \sE_{\Pic(X)} \\
 \sE_\bT \times \sE_{\Pic(X)}.
}
\end{gather*}
The coordinates on the three components of the RHS, as well as their pullbacks to $\textsf{E}_\bT(X)$, will be called \emph{Chern roots}, \emph{equivariant parameters} and \emph{K\"ahler parameters} respectively.

\subsection{GKM varieties}

For a general $X$, the equivariant elliptic cohomology $\Ell_\bT (X)$ may be difficult to describe, even if the diagram above given by Kirwan surjectivity is present. However, for the following large class of varieties called GKM varieties, it admits a nice explicit combinatorial characterization. There are many classical examples of GKM varieties, including toric varieties, hypertoric varieties, and partial flag varieties.

\begin{Definition}
 Let $X$ be a variety with a $\bT$-action. We say that $X$ is a GKM variety, if
 \begin{itemize}\itemsep=0pt
 \item $X^{\bT}$ is finite,
 \item for every two fixed points $p,q\in X^{\bT}$ there is no more than one $\bT$-equivariant curve connecting them.
 \item $X$ is $\bT$-formal, in the sense of \cite{GKM}.
 \end{itemize}
\end{Definition}

By definition, a GKM variety admits only finitely many $\bT$-fixed points and 1-dimensional $\bT$-orbits. In particular, there are finitely many $\bT$-equivariant compact curves connecting fixed points, and they are all rational curves isomorphic to $\P^1$.

By the localization theorem, we know that the irreducible components of $\Ell_\bT (X)$ are para\-meterized by fixed points $p\in X^\bT$, each isomorphic to the base $\sE_\bT$. Therefore, set-theoretically, $\Ell_\bT (X)$ is the union of $|X^{\bT}|$ copies of $\sE_{\bT}$:
\begin{gather} \label{elcoh}
\Ell_{\bT}(X)=\bigg(\coprod\limits_{p\in X^{\bT}} \Or_{p}\bigg) /\Delta,
\end{gather}
where $\Or_{p}\cong \sE_{\bT}$ and $/\Delta$ denotes the gluing data. Following \cite{AOelliptic} we will call $\Or_{p}$ the $\bT$-orbit associated to the fixed point $p$ in $\Ell_{\bT}(X)$ (even though it is not an orbit of any group action).

We have the following explicit description of $\Ell_\bT(X)$. The proof is a direct application of the characterization \cite{GKM} of $H^*_\bT(X)$ when $X$ is GKM, see \cite{RVSZ}.

\begin{Proposition} \label{progkm}
If $X$ is a GKM variety, then
 \begin{gather*}
 \Ell_{\bT}(X)=\bigg(\coprod\limits_{p\in X^{\bT}} \Or_{p}\bigg) /\Delta,
 \end{gather*}
where $/\Delta$ denotes the intersections of $\bT$-orbits $\Or_{p}$ and $\Or_{q}$ along the hyperplanes
 \begin{gather*}
 \Or_{p} \supset \chi^{\perp}_{C} \subset \Or_{q},
 \end{gather*}
for all $p$ and $q$ connected by an equivariant curve $C$, where $\chi_{C}$ is the $\bT$-character of the tangent space $T_{p} C$, and $\chi_C^\perp$ is the hyperplane in $\cE_\bT$ associated with the hyperplane $\ker \chi_C \subset \bT$. The intersections of orbits $\Or_p$ and $\Or_q$ are transversal and hence the scheme $\Ell_\bT (X)$ is a variety with simple normal crossing singularities.
\end{Proposition}

The extended version also has the same structure
\begin{gather} \label{extver}
\textsf{E}_{\bT}(X)=\bigg(\coprod\limits_{p\in X^{\bT}} \widehat\Or_{p}\bigg) /\Delta,
\end{gather}
where $\Delta$ is the same as before, and $\widehat\Or_p := \Or_p \times \sE_{\Pic(X)}$.

For each fixed point $p\in X^\bT$, we have the diagram
\begin{gather} \label{diag-Kirwan}\begin{split} &
\xymatrix{
 \widehat \Or_p \ar@{^{(}->}[r] & \textsf{E}_\bT(X) \ar[d] \ar@{^{(}->}[r] & \big( E^{\dim \bK} / W \big) \times \sE_\bT \times \sE_{\Pic(X)} \\
 & \sE_\bT \times \sE_{\Pic(X)}.
}\end{split}
\end{gather}
Let $t_1, \dots, t_{\dim \bK}$ be the elliptic Chern roots. The embedding of $\widehat\Or_p$ in $\big( E^{\dim \bK} / W \big) \times \sE_\bT \times \sE_{\Pic(X)}$ is always cut out by linear equations $t_i = t_i \big|_p$, $1\leq i\leq \dim \bK$, where $t_i \big|_p$ is a certain linear combination of equivariant parameters.

\begin{Example} Consider the $(\C^*)^{N+1}$-action on $\P^N$. The equivariant K-theory ring, viewed as a scheme, fits into the following diagram
 \begin{gather*}
 \xymatrix{
\Spec \dfrac{\C \big[L^{\pm 1}, z_1^{\pm 1}, \dots, z_{N+1}^{\pm 1}, \mu^{\pm 1} \big]}{\langle (1 - z_1 L) \cdots (1-z_{N+1} L) \rangle} \ar[d] \ar@{^{(}->}[r] & \Spec \C \big[L^{\pm 1}, z_1^{\pm 1}, \dots, z_{N+1}^{\pm 1}, \mu^{\pm 1} \big] \\
 \Spec \C \big[ z_1^{\pm 1}, \dots, z_{N+1}^{\pm 1}, \mu^{\pm 1} \big],
 }
 \end{gather*}
where $L$ is the class of $\mathcal{O}(1)$, $z_1, \dots, z_{N+1}$ are equivariant parameters, and $\mu$ is the K\"ahler parameter. Intuitively, $E_T\big(\P^N\big)$ is simply the same picture ``quotient by $q^{\Z^{N+1} \times \Z}$''. In particular, the relation $(1 - z_1 L) \cdots (1-z_{N+1} L)$ gives a simple normal crossing of $N+1$ components, each isomorphic to the base. The $i$-th component $\widehat\Or_{p_i}$, which we call orbit corresponding to the fixed point~$i$, is cut out by the linear equation $1 - z_i L = 0$.
\end{Example}

\subsection[Geometry and extended elliptic cohomology of $X$]{Geometry and extended elliptic cohomology of $\boldsymbol{X}$}\label{xdef}

From now on, let $X$ be the Nakajima quiver variety associated to the $A_{n-1}$-quiver, with dimension vector $(1, 2, \dots, n-1)$ and framing vector $(0, 0, \dots, 0, n)$. More precisely, the quiver looks like
\begin{gather*}
\xymatrix{
 V_1 \ar@<0.5ex>[r]^-{{\bf a}_1} & V_2 \ar@<0.5ex>[r]^-{{\bf a}_2} \ar@<.5ex>[l]^-{{\bf b}_1} & \cdots \ar@<.5ex>[l]^-{{\bf b}_2} \ar@<.5ex>[r]^-{{\bf a}_{n-2}} & V_{n-1} \ar@<.5ex>[l]^-{{\bf b}_{n-2}} \ar@<.5ex>[d]^-{\bf j} \\
 & & & W \ar@<.5ex>[u]^-{\bf i} ,
}
\end{gather*}
where
\begin{gather*}
V_i = \C^i, \qquad 1\leq i\leq n-1, \qquad W = \C^n.
\end{gather*}
By definition, one considers the vector space
\begin{gather*}
R = \bigoplus_{i=1}^{n-2} \Hom (V_i, V_{i+1}) \oplus \Hom (V_{n-1}, W),
\end{gather*}
acted upon naturally by $\bG := \prod\limits_{i=1}^{n-1} {\rm GL}(V_i)$, and the moment map $\mu\colon T^* R \to \prod\limits_{i=1}^{n-1} {\mathfrak{gl}}(V_i)^*$ given by
\begin{gather*}
{\bf b}_1 {\bf a}_1 = 0, \qquad {\bf a}_i {\bf b}_i - {\bf b}_{i+1} {\bf a}_{i+1} = 0, \qquad 1\leq i\leq n-3, \qquad {\bf a}_{n-2} {\bf b}_{n-2} - {\bf i} {\bf j} = 0.
\end{gather*}
Given any stability condition $\theta = (\theta_1, \dots, \theta_{n-1}) \in \mathbb{Z}^{n-1}$, there is a $\bG$-character $(g_i)_{i=1}^{n-1} \mapsto \prod\limits_{i=1}^{n-1} \left( \det g_i \right)^{\theta_i}$. We choose the stability condition to be $\theta_i <0$, $1\leq i\leq n-1$, and define
\begin{gather*}
X := \mu^{-1} (0) /\!\!/\!_\theta \bG.
\end{gather*}

\begin{Proposition} The quiver variety $X$ defined above is isomorphic to the cotangent bundle of the complete flag variety in $\C^n$.
\end{Proposition}

\begin{proof} Recall the following criterion of stability \cite{Intro}: a representative $({\bf a, b, i, j})$ is stable if and only if for any invariant subspace $S\subset V := \bigoplus_i V_i$, the following two conditions hold
\begin{enumerate}\itemsep=0pt
\item[1)] if $S \subset \ker {\bf j}$, then either $\theta \cdot \dim S >0$ or $S = 0$;
\item[2)] if $S \supset \operatorname{im} {\bf i}$, then either $\theta \cdot \dim S > \theta \cdot \dim V$ or $S = V$.
\end{enumerate}
For a representative $({\bf a, b, i, j})$ the space
 \begin{gather*}
 S = \bigoplus_{i=1}^{n-2} \ker {\bf a}_i \oplus \ker {\bf j}
 \end{gather*}
is stable under ${\bf a}$ and ${\bf b}$ by the moment map equations. Hence for the representative to be stable, it has to satisfy 1), which implies $S = 0$. In other words, ${\bf a}_i$ and ${\bf j}$ are injective, which gives a~complete flag in $\C^n$. The maps ${\bf b}_i$ then represent a~point in the cotangent fiber.
\end{proof}

Consider the torus $(\C^*)^n$ acting on $(x_1, \dots, x_n) \in W$, which descends to $X$, and an extra torus $\C^*_\hbar$ scaling the cotangent fibers{\samepage
\begin{gather*}
(x_1, \dots, x_n) \mapsto \big(x_1 z_1^{-1}, \dots, x_n z_n^{-1}\big), \qquad ({\bf a, b, i, j}) \mapsto \big({\bf a}, \hbar^{-1} {\bf b}, \hbar^{-1} {\bf i}, {\bf j} \big) ,
\end{gather*}
where $z_1, \dots, z_n, \hbar$ are the equivariant parameters.}

Let $\cV_k$, $1\leq k\leq n-1$ be the tautological bundles associated with $V_k$. Denote their Chern roots decomposition by
\begin{gather*}
\cV_k = t_1^{(k)} + \dots + t_k^{(k)}.
\end{gather*}
in the K-theory of $X$. Let $\{e_1, \dots, e_n\}$ be the standard basis of $W = \C^n$. Fixed points of $X$ are parameterized by complete flags $V_1 \subset \cdots \subset V_{n-1} \subset W$, where each $V_k$ is a coordinate subspace in $W$, i.e., spanned by a subset of size $k$ of $e_i$'s. For any $1\leq k\leq n$, let $I_k$ be the index such that $V_k / V_{k-1} = \C e_{I_k}$. Then the tuple $(I_1, \dots, I_n)$ is a permutation of the indices $(1, \dots, n)$. In other words, for each element of the symmetric group $I\in \mathfrak{S}_n$, there is a fixed point of~$X$, given by the complete flag $V_1 (I) \subset \cdots \subset V_{n-1} (I) \subset W$, where
\begin{gather*}
V_k (I) = \mathrm{Span}_\C \{e_{I_1}, \dots, e_{I_k} \}, \qquad 1\leq k\leq n.
\end{gather*}

We also introduce the notation of ordered indices:
\begin{gather} \label{ordset}
\big\{ i_1^{(k)} < \cdots < i_k^{(k)} \big\} = \{ I_1, \dots, I_k \}, \qquad 1\leq k\leq n.
\end{gather}

By Kirwan surjectivity, the extended elliptic cohomology $\textsf{E}_\bT(X)$ embeds into the space
\begin{gather*}
E \times \Sym^2 E \times \cdots \times \Sym^{n-1} E \times \sE_\bT \times \sE_{\Pic (X)}
\end{gather*}
with coordinates
\begin{gather*}
\big(t_1^{(1)}, t_1^{(2)}, t_2^{(2)}, \dots, t_1^{(n-1)}, \dots, t_{n-1}^{(n-1)}, z_1, \dots, z_n, \hbar, \mu_1, \dots, \mu_n \big).
\end{gather*}
Moreover, by the GKM description, the extended elliptic cohomology is a union of orbits
\begin{gather} \label{excx}
\textsf{E}_{\bT}(X)=\bigg(\coprod\limits_{I \in \mathfrak{S}_n} \widehat\Or_I \bigg) /\Delta,
\end{gather}
where $\widehat\Or_I$ is cut out by the linear equations
\begin{gather} \label{restr}
t_l^{(k)} = z_{i_l^{(k)}}, \qquad 1\leq l\leq k\leq n.
\end{gather}
Note that in these equations of Chern root restrictions, we have implicitly chosen an ordering of Chern roots $t_1^{(k)}, \dots, t_k^{(k)}$, depending on each fixed point.

The tangent bundle at the fixed point $I$ is
\begin{gather*}
T_I X = \sum_{1\leq l<k\leq n} \frac{z_{I_l}}{z_{I_k}} + \hbar^{-1} \sum_{1\leq l<k\leq n} \frac{z_{I_k}}{z_{I_l}}.
\end{gather*}
Choose a cocharacter of the torus $(\C)^*$
\begin{gather*}
\sigma = (1,2, \dots, n) \in \R^n,
\end{gather*}
which decomposes the tangent bundle as $T_I X = N_I^+ \oplus N_I^-$, where
\begin{gather*}
N_I^- = \sum_{\substack{1\leq l< k \leq n \\ I_l < I_k }} \frac{z_{I_l}}{z_{I_k}} + \hbar^{-1} \sum_{\substack{1\leq l < k \leq n \\ I_l > I_k }} \frac{z_{I_k}}{z_{I_l}}, \qquad N_I^+ = \sum_{\substack{1\leq l< k \leq n \\ I_l > I_k }} \frac{z_{I_l}}{z_{I_k}} + \hbar^{-1} \sum_{\substack{1\leq l < k \leq n \\ I_l < I_k }} \frac{z_{I_k}}{z_{I_l}} .
\end{gather*}

\section[Elliptic weight functions and $R$-matrices]{Elliptic weight functions and $\boldsymbol{R}$-matrices}\label{sec:2}

\subsection{Notations and parameters} \label{notesec}
Let $q\in \C^*$ be a complex number with $|q|<1$. The skew Jacobi theta function is defined by
\begin{gather*}
\thi(x)=\big(x^{1/2}-x^{-1/2}\big)\pho(qx)\>\pho(q/x),\qquad \pho(x)=\prod_{s=0}^\infty\big(1-q^s x\big).
\end{gather*}
It has the following properties
\begin{gather*}
\frac{\thi(qx)}{\thi(x)}=-\frac1{q^{1/2}x},\qquad\thi(1/x)=-\thi(x).
\end{gather*}

The elliptic weight functions depend on the following sets of parameters:
\begin{itemize}\itemsep=0pt
\item The equivariant parameters $\zz=(z_{1},\dots, z_{n})$ representing the coordinates on $\Or_I\cong \cE_{\bT}$ in~(\ref{elcoh}).
\item The K\"ahler (or dynamical) parameters $\mub=(\mu_{1},\dots, \mu_{n})$ representing the coordinates on $\cE_{\rm{Pic}(X)}$-part of the extended orbits $\widehat{\Or}_{I}$ in~(\ref{extver}).

\item The Chern roots $\ttt^{(k)} = \big(t^{(k)}_{1},\dots, t^{(k)}_{k}\big)$ of the rank~$k$ tautological bundle $\cV_k$ over $X$. We will abbreviate by $\ttt=\big(t^{(1)}_{1},\dots,t^{(n)}_{n}\big)$ the set of all Chern roots of all tautological bundles.

\item The $\bT$-equivariant weight $\hbar$ representing the weight of the symplectic form on~$X$.
\end{itemize}

For a permutation $\sigma$ we write $\zz_{\sigma}=(z_{\sigma(1)},\dots,z_{\sigma(n)})$ and $1/\zz=(1/z_1,\dots,1/z_n)$.

As we discussed in Section \ref{xdef} the fixed points $X^{\bT}$ are labeled by permutations $I=(I_1,\dots, I_n)$ of the ordered set $(1,\dots, n)$. By abuse of language we will denote the fixed point corresponding to $I$ by $I$ as well. For another permutation $\sigma \in \mathfrak{S}_n$, the product $\sigma\cdot I$ will denote the composed permutation (and also the corresponding fixed point)
\begin{gather*}
(1, \dots, n) \mapsto (\sigma (1), \dots, \sigma (n)) \mapsto (I_{\sigma (1)}, \dots, I_{\sigma (n)} ).
\end{gather*}
We \looseness=-1 will denote the restrictions of Chern roots to the orbits corresponding to fixed points (\ref{restr}) by
\begin{gather} \label{fpres}
\zz_{I}=\big(t^{(k)}_{a}=z_{i^{(k)}_a}\big) ,
\end{gather}
where $i^{(k)}_a$ are defined by (\ref{ordset}).

\subsection{Weight functions}
Let us define the elliptic weight functions
\begin{gather} \label{WI2}
W_I(\ttt,\zz,\hbar,\mub) =\Sym_{\>\ttt^{(1)}} \cdots \Sym_{\ttt^{(n-1)}} U_I(\ttt,\zz,\hbar,\mub),
\end{gather}
where the symbol $\Sym$ denotes the symmetrization over the corresponding set of variables and
\begin{gather*}
U_I(\ttt,\zz,\hbar,\mub) =\prod\limits_{k=1}^{n-1} \left(
\dfrac{\prod\limits_{a=1}^{k} \prod\limits_{c=1}^{k+1} \psi_{I,k,a,c} \left( \dfrac{t^{(k+1)}_c}{t^{(k)}_a} \right) }{\prod\limits_{1 \leq a<b \leq k } \thi \left( \dfrac{t^{(k)}_a \hbar}{t^{(k)}_b} \right) \thi \left( \dfrac{t^{(k)}_b}{t^{(k)}_a} \right) }\right)
\end{gather*}
with convention $t^{(n)}_{i}=z_{i}$ and
\begin{gather*}
\psi_{I,k,a,c}(x) = \begin{cases}
 \thi(\hbar x) ,& \text{if} \ i^{(k+1)}_c<i^{(k)}_a,\\
 \displaystyle
 {\thi \left( \dfrac{x\hbar^{1 -p_{I,k+1}(i_a^{(k)})}
 \mu_{k+1} }{ \mu_{j(I,k,a)} } \right)} , &
 \text{if} \ i^{(k+1)}_c=i^{(k)}_a, \\
 \thi(x),& \text{if} \ i^{(k+1)}_c>i^{(k)}_a .
\end{cases}
\end{gather*}
Here the index $j(I,k,a) \in \{1, \dots, n \}$ is defined such that
\begin{gather*}
I_{j (I, k, a)} = i_a^{(k)},
\end{gather*}
and
\begin{gather} \label{p-function}
p_{I, j} (m) = \begin{cases}
 1, & I_j <m, \\
 0 , & I_j \geq m.
\end{cases}
\end{gather}

For a permutation $\sigma\in \mathfrak{S}_n$ we also define the elliptic weight function
\begin{gather*}
W_{\sigma,I}(\ttt,\zz,\hbar,\mub) :=W_{\sigma^{-1}(I)}(\ttt,\zz_{\sigma},\hbar,\mub).
\end{gather*}
Of particular importance will be the weight function corresponding to the longest permutation $\sigma_0=(n,n-1,\dots,2,1) \in \mathfrak{S}_n$.

Define
\begin{gather} \label{resmat}
A^{\sigma}_{I,J}(\zz,\mub) = W_{\sigma, I} (\zz_{J},\zz,h,\mub),
\end{gather}
the matrix of restrictions of elliptic weight functions to fixed points. For $\sigma=\id$ we will abbreviate it to $A_{I,J}(\zz,\mub)$.

\subsection{Properties of weight functions and restriction matrices}
The elliptic weight functions enjoy several interesting combinatorial identities. Here we list some of them which will be used below. A more detailed exposition can be found in~\cite{Rim,RTV}.

Let us set
\begin{gather*}
P_{I}(z_1,\dots,z_n)=\prod\limits_{{1\leq k <l\leq n}\atop {I_{l}<I_{k}}} \thi\left(\frac{\hbar z_{I_{l}}}{z_{I_{k}}}\right)
\prod\limits_{{1\leq k <l\leq n}\atop {I_{l}>I_{k}}} \thi\left(\frac{ z_{I_{l}}}{z_{I_{k}}}\right).
\end{gather*}
This function satisfies the following property:
\begin{Lemma} \label{pprop}
\begin{gather*}
P_{\sigma_0 \cdot I \cdot \sigma_{0}}\big(z^{-1}_{\sigma_{0}(1)},\dots, z^{-1}_{\sigma_{0}(n)}\big)=P_{I}(z_1,\dots,z_n).
\end{gather*}
\end{Lemma}
\begin{proof}By direct computation.
\end{proof}

\begin{Lemma} \label{hollem}For the dominance order on permutations, the matrix $A_{I,J}(\zz,\mub)$ is lower triangular, i.e.,
\begin{gather*}
A_{I,J}(\zz,\mub)=0, \qquad {\rm if} \quad J \succ I
\end{gather*}
and the diagonal elements are given by
\begin{gather} \label{diagonal}
A_{I,I}(\zz,\mub)=(-1)^{I} P_{I}(z_1,\dots,z_n) P_{I^{-1}\cdot \sigma_{0}} (\mu_{\sigma_{0}(1)},\dots,\mu_{\sigma_{0}(n)}),
\end{gather}
where $(-1)^{I}$ stands for the parity of the permutation $I$. The matrix functions $A_{I,J}(\zz,\mub)$ are holomorphic in all variables $\zz$, $\hbar$, $\mub$.
\end{Lemma}
\begin{proof}Lemmas 2.4, 2.5 and 2.6 in \cite{RTV}.
\end{proof}

Let us consider the elliptic dynamical $R$-matrix in the Felder's normalization
\begin{gather*}
R^{j,j}_{j,j}(x,\mub)=1, \qquad R^{j,k}_{j,k}(x,\mub)=\dfrac{\thi(x) \thi \left( \dfrac{\hbar \mu_j}{\mu_{k}} \right)}{\thi(x \hbar) \thi \left( \dfrac{\mu_j}{\mu_k} \right)}, \qquad R^{j,k}_{k,j}(x,\mub)=\dfrac{\thi \left( \dfrac{x \mu_j}{\mu_k} \right) \thi(\hbar)}{\thi(x \hbar) \thi \left( \dfrac{\mu_j}{\mu_k} \right)} ,
\end{gather*}
where $1\leq j, k\leq n$, $j\neq k$.

\begin{Lemma}The weight functions \eqref{WI2} satisfy the following recursive relations
\begin{gather*}
W_{I \cdot s_{k}}^{z_k\leftrightarrow z_{k+1}}=R^{a,b}_{a,b} \left( \dfrac{z_k}{z_{k+1}} \right) W_{I}+R^{b,a}_{a,b} \left( \dfrac{z_k}{z_{k+1}} \right) W_{I \cdot s_k},
\end{gather*}
where $a := I^{-1} (k)$, $b := I^{-1} (k+1)$, and $s_k$ denotes the transposition $(k,k+1)$. The superscript $z_k\leftrightarrow z_{k+1}$ denotes the function in which $z_k$ is substituted by $z_{k+1}$ and $z_{k+1}$ by $z_{k}$.
\end{Lemma}
\begin{proof}Theorem 2.2 in \cite{RTV}.
\end{proof}

We can reformulate those as relations among the matrix elements of the restriction matrix.

\begin{Corollary} \label{cor1}The elements of the restriction matrix satisfy the following relations:
\begin{gather} \label{rmatrel}
A_{I\cdot s_{k} ,J \cdot s_{k}}(\zz,\mub)^{z_{k}\leftrightarrow z_{k+1}}=R^{a,b}_{a,b} \left( \dfrac{z_k}{z_{k+1}} \right) A_{I,J}(\zz,\mub) + R^{b,a}_{a,b} \left( \dfrac{z_k}{z_{k+1}} \right) A_{I \cdot s_k, J}(\zz,\mub).
\end{gather}
\end{Corollary}The identity \eqref{rmatrel} can be used to compute recursively all matrix elements $A_{I,J}(\zz,\mub)$ from the known diagonal entries~(\ref{diagonal}):
\begin{Lemma} \label{reclem}
The restriction matrix $A_{I,J}(\zz,\mub)$ is the unique lower triangular matrix $($in the basis of indexes $I$ ordered by $\succ)$ with the diagonal elements given by~\eqref{diagonal} satisfying the $R$-matrix relations~\eqref{rmatrel}.
\end{Lemma}
\begin{proof}The proof is by induction on rows of the restriction matrix. The restriction matrix $A_{I,J}(\zz,\mub)$ is lower triangular if~$I$,~$J$ are ordered
by the dominance order $\succ$. Thus, the only nontrivial matrix element in the first row is $A_{\id,\id}(\zz,\mub)$. This matrix element is fixed by (\ref{diagonal}) and thus all elements in the first row are uniquely determined.

Note that (\ref{rmatrel}) can be rewritten as:
\begin{gather*}
A_{I\cdot s_k,J \cdot s_k}(\zz,\mub)= \alpha_{s_k} A_{I,J}(\zz,\mub)^{z_{k}\leftrightarrow z_{k+1}}+\beta_{s_k} A_{I,J \cdot s_k}(\zz,\mub)
\end{gather*}
for certain explicit functions $\alpha_{s_k}$ and $\beta_{s_k}$. For any $I' \neq \id$, there always exists some $k$, such that for $I := I' \cdot s_k$, we have $I' = I\cdot s_k \succ I$. Thus, the last identity is the expression for matrix elements in the $I'$-th row in terms of its values in the previous rows. The result follows by induction.
\end{proof}

\subsection[Dual $R$-matrix relations]{Dual $\boldsymbol{R}$-matrix relations}
Recent results in \cite{RW} show that the matrix elements of the restriction matrices satisfy another recursion, named ``Bott--Samelson recursion'' in~\cite{RW}. We will call this other recursion the ``dual $R$-matrix relations'' and explain later that these relations correspond to $R$-matrix relations on the symplectic dual variety~$X'$.

\begin{Theorem} \label{duthm}The elements of the restriction matrix satisfy the following relations
\begin{gather} \label{dualrel}
A_{s_k \cdot I, s_k \cdot J} (\zz, \mub)^{\mu_k\leftrightarrow \mu_{k+1}}=\tilde{R}^{a,b}_{a,b} A_{I,J}(\zz,\mub) + \tilde{R}^{b,a}_{a,b} A_{I,s_k \cdot J}(\zz,\mub),
\end{gather}
where $a=n-J_k+1$ and $b=n-J_{k+1}+1$ and the coefficients $\tilde{R}^{a,b}_{c,d}$ are related to the coefficients of Felder's $R$-matrix by
\begin{gather} \label{dualR}
\tilde{R}^{a,b}_{c,d}=\left.{R}^{a,b}_{c,d}\right|_{z_i \mapsto \mu_i^{-1} ,\, \mu_i \mapsto z_{\sigma_{0}(i)}}.
\end{gather}
\end{Theorem}
\begin{proof}This identity is equivalent to Theorem~11.1 in~\cite{RW}. Indeed, direct computations show that the weight functions
${\bf w}_{I}$ used in \cite{RW} differ from the one used in the present paper by a~factor
\begin{gather*}
W_I = {\bf w}_{I} \cdot C \left( \dfrac{\thi (\hbar)}{\thi' (1)} \right)^{\sharp \{ (i,j) \,|\, 1\leq i<j\leq n, \, I_i > I_j \} } \prod_{1\leq i<j\leq n} \thi \left( \hbar^{1- p_{j-1} (i)} \dfrac{\mu_j}{\mu_i} \right) ,
\end{gather*}
where $C$ a constant independent of $I$, and $p_{j-1} (i)$ is given by~(\ref{p-function}). Substituting this to the equation~(33) of~\cite{RW}, we arrive at~(\ref{dualrel}).
\end{proof}

The following Lemma and its proof is analogous to Lemma~\ref{reclem}.
\begin{Lemma} \label{dulem} The restriction matrix $A_{I,J}(\zz,\mub)$ is the unique lower triangular matrix $($in the basis of indexes $I$ ordered by $\succ)$ with diagonal elements given by~\eqref{diagonal} satisfying the recursive relations~\eqref{dualrel}.
\end{Lemma}

\begin{Note} We found that the matrix elements $A_{I,J}(\zz,\mub)$ can be computed in two different ways: using recursion (\ref{rmatrel}) or recursion~(\ref{dualrel}). This fact provides a set of highly nontrivial identities for elliptic functions. We give several examples of these identities in Section~\ref{exsection}, see also~\cite[Section 9]{RW}. In general, these identities can be formulated as Theorem \ref{duthrm} below.
\end{Note}

The recursive relations (\ref{rmatrel}) and (\ref{dualrel}) are closely related:

\begin{Proposition} \label{prop1} Let $A_{I,J}(\zz,\mub)$ be a matrix satisfying relations~\eqref{rmatrel}. Let $B_{I,J}(\zz,\mu)$ be the matrix defined by
\begin{gather} \label{invtrans}
B_{I,J}(z_1,\dots, z_n,\mu_1,\dots,\mu_n)=A_{\sigma_{0}\cdot J^{-1},\sigma_{0}\cdot I^{-1}}\big(\mu_1^{-1},\dots,\mu_{n}^{-1},z_{\sigma_{0}(1)},\dots,z_{\sigma_{0}(n)}\big).
\end{gather}
Then the matrix $B_{I,J}(\zz,\mub)$ satisfies the relations~\eqref{dualrel}.
\end{Proposition}
\begin{proof}Expressing $A_{I,J}(\zz,\mub)$ from (\ref{invtrans}), we find
\begin{gather*}
A_{I,J}(z_1,\dots,z_{n},\mu_{1},\dots,\mu_{n}) = B_{J^{-1}\cdot \sigma_{0},I^{-1}\cdot \sigma_{0}}\big(\mu_{\sigma_{0}(1)},\dots,\mu_{\sigma_{0}(n)},z^{-1}_1,\dots, z_{n}^{-1}\big).
\end{gather*}
Substituting this into (\ref{rmatrel}) we obtain
\begin{gather*}
B_{s_k\cdot J^{-1} \cdot \sigma_{0},s_k\cdot I^{-1} \cdot \sigma_{0}}\big(\mu_{\sigma_{0}(1)},\dots,\mu_{\sigma_{0}(n)},z_{1}^{-1},\dots, z_{n}^{-1}\big)^{z_{k}\leftrightarrow z_{k+1}}\\
\qquad {}= R^{a,b}_{a,b} B_{J^{-1}\cdot \sigma_{0},I^{-1} \cdot \sigma_{0}}\big(\mu_{\sigma_{0}(1)},\dots,\mu_{\sigma_{0}(n)},z_{1}^{-1},\dots, z_{n}^{-1}\big)\\
\qquad \quad {} + R^{b,a}_{a,b} B_{J^{-1}\cdot \sigma_{0},s_k \cdot I^{-1} \cdot \sigma_{0}}\big(\mu_{\sigma_{0}(1)},\dots,\mu_{\sigma_{0}(n)},z_{1}^{-1},\dots, z_{n}^{-1}\big).
\end{gather*}
To see that this identity is equivalent to~(\ref{dualrel}), we change the indices of the matrices by
\begin{gather} \label{DualPoint}
J^{-1}\cdot \sigma_{0}\mapsto I, \qquad I^{-1}\cdot \sigma_{0}\mapsto J,
\end{gather}
such that
\begin{gather*}
 B_{s_k \cdot I ,s_k\cdot J }\big(\mu_{\sigma_{0}(1)},\dots,\mu_{\sigma_{0}(n)},z_{1}^{-1},\dots, z_{n}^{-1}\big)^{z_{k}\leftrightarrow z_{k+1}}\\
 \qquad{} =
 R^{a,b}_{a,b} B_{I,J}\big(\mu_{\sigma_{0}(1)},\dots,\mu_{\sigma_{0}(n)},z_{1}^{-1},\dots, z_{n}^{-1}\big)\\
 \qquad\quad{}+R^{b,a}_{a,b} B_{I,s_k \cdot J}\big(\mu_{\sigma_{0}(1)},\dots,\mu_{\sigma_{0}(n)},z_{1}^{-1},\dots, z_{n}^{-1}\big).
\end{gather*}
Substitution $z_{i} \mapsto \mu_{i}^{-1}$, $\mu_{i} \mapsto z_{\sigma_{0}(i)}$ simplifies it to
\begin{gather} \label{beque}
B_{s_k \cdot I ,s_k\cdot J }(\zz,\mub)^{\mu_{k}\leftrightarrow \mu_{k+1}}=\tilde{R}^{a,b}_{a,b} B_{I,J}(\zz,\mub)+\tilde{R}^{b,a}_{a,b} B_{I,s_k \cdot J}(\zz,\mub),
\end{gather}
where $\tilde{R}^{b,a}_{a,b}$ are related to Felder's $R$-matrix as in (\ref{dualR}). Finally, in $R$-matrix relations (\ref{rmatrel}) the index $a$ is the number of the element $k$ in the permutation $I$ and $b$ is the number of the element $k+1$ in $I$. After changing indexes as in~(\ref{DualPoint}) we find that $a=n-J_k+1$ and $b=n-J_{k+1}+1$. We see that relation~(\ref{beque}) coincides with~(\ref{dualrel}).
\end{proof}

We conclude the following result.
\begin{Theorem} \label{duthrm} The elements of the restriction matrix satisfy the following identities
\begin{gather} \label{mainsym}
A_{I,J}(\zz,\mub)=(-1)^{n(n-1)/2} A_{J^{-1}\cdot \sigma_{0},I^{-1}\cdot \sigma_{0}}\big(\mu_{\sigma_{0}(1)},\dots,\mu_{\sigma_{0}(n)},z^{-1}_1,\dots, z_{n}^{-1}\big),
\end{gather}
where $\sigma_{0}$ denotes the longest permutation in the symmetric group $\mathfrak{S}_{n}$.
\end{Theorem}

\begin{proof}Let $B_{I,J}(\zz,\mu)$ be as in the previous proposition. First \begin{gather*} A_{I,I}(\zz,\mu)=(-1)^{n(n-1)/2} B_{I,I}(\zz,\mu).\end{gather*} This follows from~(\ref{diagonal}) Lemma~\ref{pprop} and $(-1)^{\sigma_{0}}=(-1)^{n(n-2)/2}$.

By Corollary \ref{cor1} $A_{I,J}(\zz,\mu)$ satisfies the $R$-matrix relations, and thus by the previous proposition $B_{I,J}(\zz,\mu)$ satisfies relations~(\ref{dualrel}).
By Lemmas~\ref{reclem} and~\ref{dulem}, we conclude
\begin{gather*}
A_{I,J}(\zz,\mu)=(-1)^{n(n-1)/2} B_{I,J}(\zz,\mu).
\end{gather*}
This identity is equivalent to (\ref{mainsym}) after the change of variables $z_{i}\mapsto \mu_{\sigma_{0}(i)}$, $\mu_{i}\mapsto z_i^{-1}$ and indexes $\sigma_{0}\cdot J^{-1}\mapsto I$, $\sigma_{0}\cdot I^{-1}\mapsto J$.
\end{proof}

\begin{Note}
We would like to stress here that the identity (\ref{mainsym}) describes a symmetry between two sets of parameters of completely different nature: the
equivariant parameters $\zz$ and the K\"ahler parameters $\mub$. The symmetry of the elliptic stable envelopes with respect to the transformation $\zz\leftrightarrow \mub$ is one of the predictions of 3d mirror symmetry. We will discuss this point of view in Section~\ref{envelsection}.
\end{Note}

\subsection{Examples} \label{exsection}

\textbf{Case $n=2$.} Using (\ref{WI2}) we find that the weight functions are equal
\begin{gather*}
W_{(1,2)}=\thi \left( {\frac {\hbar z_{{1}}\mu_{{2}}}{t_{{1}}^{(1)} \mu_{{1}}}} \right)
\thi \left( {\frac {z_{{2}}}{t_1^{(1)} }} \right), \qquad W_{(2,1)}=\thi \left( {\frac {\hbar z_{{1}}}{t_1^{(1)} }} \right) \thi \left( {
 \frac {z_{{2}}\mu_{{2}}}{t_1^{(1)} \mu_{{1}}}} \right).
\end{gather*}
Here, as we defined in Section~\ref{notesec}, $(1,2)$ and $(2,1)$ denote the fixed points corresponding to the trivial and non-trivial permutations of $\mathfrak{S}_{2}$ respectively.

By~(\ref{fpres}) the restriction to the point $(1,2)$ is given by the substitution $t_1^{(1)} =z_{{1}}$ and that to the point $(2,1)$ is given by the substitution $t_1^{(1)} =z_{{2}}$. Thus, in the basis of permutations ordered by $(1,2)$, $(2,1)$, the matrix of restrictions equals
\begin{gather*}
A_{I,J}(z_1,z_2,\mu_1,\mu_2)=
\left( \begin{matrix} \thi \left( {\dfrac {\hbar \mu_{{2}}}{\mu_{{1}
}}} \right) \thi \left( {\dfrac {z_{{2}}}{z_{{1}}}} \right) & 0
\vspace{2mm}\\ \thi( \hbar ) \thi \left( {\dfrac {z_
 {{2}}\mu_{{2}}}{z_{{1}}\mu_{{1}}}} \right) & \thi \left( {\dfrac {\hbar z_
 {{1}}}{z_{{2}}}} \right) \thi \left( {\dfrac {\mu_{{2}}}{\mu_{{1}}}} \right) \end{matrix} \right).
\end{gather*}
The statement of Theorem \ref{duthrm} in this case is equivalent to the following system of identities
\begin{gather*}
A_{(1, 2), (1, 2)}(z_1,z_2,\mu_1,\mu_2) =- A_{(2, 1), (2, 1)}(\mu_2,\mu_1,1/z_1,1/z_2),\\
A_{(1, 2), (2, 1)}(z_1,z_2,\mu_1,\mu_2) =- A_{(1, 2), (2, 1)}(\mu_2,\mu_1,1/z_1,1/z_2),\\
A_{(2,1), (1, 2)}(z_1,z_2,\mu_1,\mu_2) =- A_{(2, 1), (1, 2)}(\mu_2,\mu_1,1/z_1,1/z_2),\\
A_{(2, 1), (2, 1)}(z_1,z_2,\mu_1,\mu_2) =- A_{(1, 2), (1, 2)}(\mu_2,\mu_1,1/z_1,1/z_2).
\end{gather*}
It is easy to observe that all these identities trivially follow from $\thi (1/x)=-\thi (x)$. The situation, however, is more involved in the ``non-abelian'' cases $n\geq 3$.

\textbf{Case $\boldsymbol{n=3}$.} In this case one checks that the identities (\ref{mainsym}) are all trivial (i.e., both sides are equal to zero or coincide trivially) except the following matrix elements
\begin{gather*}
A_{(3, 1, 2), (1, 2, 3)}(z_1,z_2,z_3,\mu_1,\mu_2,\mu_3) = - A_{(3, 2, 1), (2, 1, 3)}(\mu_3,\mu_2,\mu_1,1/z_1,1/z_2,1/z_3),\\
A_{(3, 2, 1), (2, 1, 3)}(z_1,z_2,z_3,\mu_1,\mu_2,\mu_3) = -A_{(2, 3, 1), (1, 2, 3)}(\mu_3,\mu_2,\mu_1,1/z_1,1/z_2,1/z_3),\\
A_{(3, 2, 1), (1, 2, 3)}(z_1,z_2,z_3,\mu_1,\mu_2,\mu_3) =- A_{(3, 2, 1), (1, 2, 3)}(\mu_3,\mu_2,\mu_1,1/z_1,1/z_2,1/z_3),
\end{gather*}
Let us, for instance, compute the two sides of the last line. Using the definition (\ref{WI2}) we have
\begin{gather*}
W_{(3, 2, 1)}(\ttt,\zz,\hbar,\mub)\\
\qquad{}= \dfrac{\thi\! \left( \dfrac {\hbar t_1^{(2)}}{t_1^{(1)}} \right) \thi\! \left( {
 \dfrac {t_2^{(2)} \mu_{{2}}}{t^{(1)}_1 \mu_{{1}}}} \right) \thi\! \left(
 {\dfrac {\hbar z_{{1}}}{t^{(2)}_1 }} \right) \thi\! \left( {\dfrac {z_{{2}} \mu
 _{{3}}}{t^{(2)}_1 \mu_{{2}}}} \right) \thi\! \left( {\dfrac {z_{{3}}}{t^{(2)}_1 }} \right) \thi\! \left( {\dfrac {\hbar z_{{1}}}{t^{(2)}_2 }} \right)
 \thi\! \left( {\dfrac {\hbar z_{{2}}}{t^{(2)}_2 }} \right) \thi\! \left( {
 \dfrac {z_{{3}}\mu_{{3}}}{t^{(2)}_2 \mu_{{1}}}} \right)
}{\thi\! \left( {\dfrac {\hbar t^{(2)}_1 }{t^{(2)}_2 }} \right) \thi\! \left( { \dfrac {t^{(2)}_2 }{t^{(2)}_1 }} \right)} \\
\qquad\quad{} + \big( t^{(2)}_1 \leftrightarrow t^{(2)}_2 \big).
\end{gather*}
where the second term $\big( t^{(2)}_1 \leftrightarrow t^{(2)}_2 \big)$ denotes the first term with $t^{(2)}_1$, $t^{(2)}_2$ switched.

By (\ref{fpres}), the restriction of a weight function to $(3,2,1)$ corresponds to the specialization $t^{(1)}_1 =z_{{1}}$, $t^{(2)}_1 =z_{{1}}$, $t^{(2)}_2=z_{{2}}$. Thus, we compute
\begin{gather*}
A_{(3, 2, 1), (1, 2, 3)}(z_1,z_2,z_3,\mu_1,\mu_2,\mu_3) \\
\qquad{}= -\dfrac{\thi ( \hbar ) ^{3}\thi \left( {\dfrac {z
 _{{1}}\mu_{{1}}}{z_{{2}}\mu_{{2}}}} \right) \thi \left( {\dfrac {z_{{
 1}}\mu_{{2}}}{z_{{2}}\mu_{{3}}}} \right) \thi \left( {\dfrac {z_{{1}}
 }{z_{{3}}}} \right) \thi \left( {\dfrac {z_{{2}}\mu_{{1}}}{z_{{3}}\mu
 _{{3}}}} \right)}{\thi \left( {\dfrac {z_{{1}}}{z_{{2}}}} \right)} \\
 \qquad\quad{} +
 \dfrac{\thi( \hbar ) \thi \left( {
 \dfrac {\hbar z_{{1}}}{z_{{2}}}} \right) \thi \left( {\dfrac {\mu_{{1}}}{
 \mu_{{2}}}} \right) \thi \left( {\dfrac {z_{{2}}}{z_{{3}}}} \right)
 \thi \left( {\dfrac {\mu_{{2}}}{\mu_{{3}}}} \right) \thi \left( {
 \dfrac {z_{{1}}\mu_{{1}}}{z_{{3}}\mu_{{3}}}} \right) \thi \left( {
 \dfrac {\hbar z_{{2}}}{z_{{1}}}} \right)}{\thi \left( {\dfrac {z_{{1
 }}}{z_{{2}}}} \right)},
\end{gather*}
and the identity above takes the form
\begin{gather*}
-\dfrac{\thi ( \hbar ) ^{3} \thi \left( {\dfrac {z
 _{{1}}\mu_{{1}}}{z_{{2}}\mu_{{2}}}} \right) \thi \left( {\dfrac {z_{{
 1}}\mu_{{2}}}{z_{{2}}\mu_{{3}}}} \right) \thi \left( {\dfrac {z_{{1}}
 }{z_{{3}}}} \right) \thi \left( {\dfrac {z_{{2}}\mu_{{1}}}{z_{{3}}\mu
 _{{3}}}} \right) }{\thi \left( {\dfrac {z_{{1}}}{z_{{2}}}} \right)} \\
 \qquad\quad{} +
\dfrac{\thi( \hbar ) \thi \left( {
 \dfrac {\hbar z_{{1}}}{z_{{2}}}} \right) \thi \left( {\dfrac {\mu_{{1}}}{
 \mu_{{2}}}} \right) \thi \left( {\dfrac {z_{{2}}}{z_{{3}}}} \right)
 \thi \left( {\dfrac {\mu_{{2}}}{\mu_{{3}}}} \right) \thi \left( {
 \dfrac {z_{{1}}\mu_{{1}}}{z_{{3}}\mu_{{3}}}} \right) \thi \left( {
 \dfrac {\hbar z_{{2}}}{z_{{1}}}} \right) }{\thi \left( {\dfrac {z_{{1
 }}}{z_{{2}}}} \right)} \\
\qquad{}=
 - \dfrac{\thi(\hbar)^{3} \thi \left( {\dfrac {z
 _{{2}}\mu_{{3}}}{z_{{1}}\mu_{{2}}}} \right) \thi \left( {\dfrac {z_{{
 3}}\mu_{{3}}}{z_{{2}}\mu_{{2}}}} \right) \thi \left( {\dfrac {\mu_{{3
 }}}{\mu_{{1}}}} \right) \thi \left( {\dfrac {z_{{3}}\mu_{{2}}}{z_{{1}
 }\mu_{{1}}}} \right) }{ \thi \left( {\dfrac {\mu_{{3}}}{\mu_{{2}
 }}} \right) }\\
 \qquad\quad{} +\dfrac{\thi(\hbar) \thi \left( {
 \dfrac {\hbar\mu_{{3}}}{\mu_{{2}}}} \right) \thi \left( {\dfrac {z_{{2}}}{
 z_{{1}}}} \right) \thi \left( {\dfrac {\mu_{{2}}}{\mu_{{1}}}}
 \right) \thi \left( {\dfrac {z_{{3}}}{z_{{2}}}} \right) \thi
 \left( {\dfrac {z_{{3}}\mu_{{3}}}{z_{{1}}\mu_{{1}}}} \right) \thi
 \left( {\dfrac {\hbar \mu_{{2}}}{\mu_{{3}}}} \right) }{ \thi \left(
 {\dfrac {\mu_{{3}}}{\mu_{{2}}}} \right) } .
\end{gather*}
This is an example of nontrivial identity satisfied by the Jacobi theta-functions. It is equivalent to the so called four-term identity for the theta functions, see equation~(2.7) in~\cite{RTV}, after some identification of the parameters.

\section{Elliptic stable envelopes \label{envelsection}}
\subsection{Elliptic stable envelopes in holomorphic normalization}
The elliptic stable envelopes for Nakajima quiver varieties were defined in~\cite{AOelliptic}. If $X$ is the Nakajima quiver variety defined in Section~\ref{xdef} (the cotangent bundle over the full flag variety) and $I\in X^{\bT}$ is a fixed point then the elliptic stable envelope ${\bf Stab}_{\sigma}(I)$ is the unique section of a certain line bundle over $\textsf{E}_{\bT}(X)$ distinguished by a set of remarkable properties. We refer to \cite[Section~3]{AOelliptic} for the original definition. The elliptic stable envelope depends on a choice of a chamber~$\sigma$. For~$X$ the set of chambers coincides with the set of Weyl chambers of the Lie algebra~$\mathfrak{sl}_{n}$ and thus, the chambers are parameterized by permutations~$\sigma$, see~\cite{Rim} for the detailed discussion of cotangent bundles over partial flag varieties.

Let us set $\mathsf{S}(X)=\prod\limits_{k=1}^{n-1} S^k E$ where $S^k E$ denotes the $k$-th symmetric power of the elliptic curve~$E$. Coordinates on $\mathsf{S}(X)$ are symmetric functions in Chern roots $\ttt$ of the tautological bundles. Recall the following map as in~(\ref{diag-Kirwan})
\begin{gather*}
\textsf{E}_{\bT}(X) \stackrel{c_{X}}{\longrightarrow} \mathsf{S}(X)\times \cE_{\bT} \times \cE_{{\rm Pic}(X)},
\end{gather*}
given by the elliptic Chern classes of the tautological bundles over $X$. It is known that $c_{X}$ is an embedding \cite{McGN}, see also \cite[Section~2.4]{AOelliptic}.

The elliptic weight functions $W_{\sigma, I}(\ttt,\zz,\hbar,\mu)$ are symmetric in $\ttt$ and thus represent sections of certain line bundles over the scheme $\mathsf{S}(X)\times \cE_{\bT} \times \cE_{{\rm Pic}(X)}$. The following theorem describes the known relation between the weight functions and the elliptic stable envelopes for $X$.
\begin{Theorem}The elliptic stable envelope of a fixed point $I\in X^{\bT}$ for a chamber $\sigma$ is given by the restriction of the corresponding elliptic weight function to elliptic cohomology of $X$:
\begin{gather} \label{envweit}
\textbf{Stab}_{\sigma}(I)=c_{X}^{*}W_{\sigma, I}(\ttt,\zz,\hbar,\mub).
\end{gather}
\end{Theorem}
\begin{proof}In the original paper \cite{AOelliptic} the elliptic stable envelope $\textbf{Stab}_{\sigma}(I)$ was defined as the unique section of certain line bundle satisfying a list of defining conditions. It was checked in Theorem~7.3 of~\cite{RTV} that the right side of~(\ref{envweit}) satisfies these conditions.
\end{proof}

\begin{Remark} The elliptic stable envelopes $\Stab^{AO}_{\sigma}(I)$ defined by Aganagic--Okounkov in \cite{AOelliptic} and the restrictions~(\ref{envweit}) differ by a normalization (i.e., by a factor). One of the defining properties in \cite{AOelliptic} fixes the diagonal restriction
\begin{gather*}
 \Stab^{AO}_{\sigma}(I)\big|_{\widehat{\Or}_{I}}=P_{\sigma^{-1} \cdot I}(\zz_{\sigma}),
\end{gather*}
while in our normalization of the elliptic weight functions the diagonal restrictions are given by~(\ref{diagonal}). This means that the Aganagic--Okounkov stable envelopes and the ones we use in the present paper are related by
\begin{gather*}
\textbf{Stab}_{\sigma}(I)=(-1)^{\sigma^{-1} I} P_{I^{-1} \sigma \sigma_{0}}(\mu_{\sigma_0(1)},\dots,\mu_{\sigma_0(n)}) \Stab^{AO}_{\sigma}(I).
\end{gather*}
That is, the two versions of stable envelopes are sections of line bundles related by the twist of a~line bundle which $P_{I^{-1} \sigma \sigma_{0}}(\mu_{\sigma_0(1)}, \dots,\mu_{\sigma_0(n)})$ is a section of. We chose to use~(\ref{envweit}) is this paper because in this normalization the stable envelopes are \textit{holomorphic}, see Lemma~\ref{hollem}.
\end{Remark}

\subsection[Dual variety $X'$ and dual stable envelope]{Dual variety $\boldsymbol{X'}$ and dual stable envelope}
 Let us fix a second copy of symplectic variety isomorphic to the cotangent bundle over the full flag variety. To distinguish it from $X$ we denote it by $X'$. We will refer to $X'$ as ``dual variety''. We denote the torus acting on $X'$ by $\bT'$ (by definition, it acts on $X'$ in the same way the torus~$\bT$ acts on~$X$). As in~(\ref{excx}) the extended equivariant elliptic cohomology scheme of this variety has the following form
\begin{gather} \label{duext}
\textsf{E}_{\bT'}(X')=\bigg(\coprod\limits_{I\in (X')^{\bT'}} \widehat{\Or}'_{I}\bigg)/\Delta,
\end{gather}
 where $\widehat{\Or}'_{I}\cong \cE_{\bT'} \times \cE_{{\rm Pic}(X')}$. We will denote by $(\zz',\hbar',\mub')$ the coordinates on $\widehat{\Or}'_{I}$.

We denote by ${\bf Stab}'$ the elliptic stable envelope for the dual variety corresponding to the chamber $\sigma_{0}$:
 \begin{gather} \label{duenv}
 {\bf Stab}' (I)=(-1)^{n(n-1)/2} c_{X'}^{*} W_{\sigma_{0},I}(\ttt',\zz',\hbar',1/\mub'),
 \end{gather}
where $\ttt'$ stands for the set of Chern roots of the tautological bundles over $X'$ and $c_{X'}$ is the same as in the previous subsection.

 \subsection{Identification of K\"ahler and equivariant parameters}

 Although as varieties $X$ and $X'$ are isomorphic, we treat them differently. In particular, fixed points and parameters will be identified in a nontrivial way.

 We fix an isomorphism of extended orbits of dual varieties
 \begin{gather*}
 \kappa \colon \ \widehat{\Or}_{I} \rightarrow \widehat{\Or}'_{J}
 \end{gather*}
 defined explicitly in coordinates by
 \begin{gather*}
 \mu_{i}' \mapsto z_{i}, \qquad z_{i}' \mapsto \mu_{i}, \qquad \hbar' \mapsto \hbar , \qquad i=1,\dots, n.
 \end{gather*}
 Note that $\kappa$ {\it maps the equivariant parameters of $X$ to the K\"ahler parameters of $X'$ and vice versa}. In particular, it provides an isomorphisms (which we denote by the same symbol, for simplicity)
 \begin{gather} \label{ident}
 \kappa\colon \ \cE_{{\rm Pic}(X)}\cong \cE_{\bT'}, \qquad \cE_{{\rm Pic}(X')}\cong \cE_{\bT}.
 \end{gather}

\subsection{3d mirror symmetry of cotangent bundles over full flag varieties}
 It is clear that $X^{\bT}$ and $(X')^{\bT'}$ are the same sets. We define a bijection
 \begin{gather*}
 \textsf{bj}\colon \ X^{\bT} \rightarrow (X')^{\bT'}, \qquad \textsf{bj}(I) := I^{-1}.
 \end{gather*}
 We say that $J \in (X')^{\bT'}$ is the fixed point corresponding to a fixed point $I\in X^{\bT}$ if $J=\textsf{bj}(I)$.

Now we are ready to formulate our main theorem revealing $z\leftrightarrow \mu$ symmetry of elliptic stable envelopes associated with the cotangent bundles over full flag varieties:
 \begin{Theorem} Let $I,J \in X^{\bT}$ be fixed points and $I^{-1}$, $J^{-1}$ be the corresponding fixed points on the dual variety. Then
\begin{gather*} 
 {\bf Stab}(I)\big|_{\widehat{\Or}_{J}}=\kappa^{*}\big( {\bf Stab}'\big(J^{-1}\big)\big|_{\widehat{\Or}^{'}_{I^{-1}}}\big).
\end{gather*}
 \end{Theorem}
 \begin{proof} By definition ${\bf Stab}(I)|_{\widehat{\Or}_{J}}=A_{I,J}(\zz,\mub)$. Similarly, by (\ref{duenv}) we have
 \begin{gather*}
 {\bf Stab}'\big(J^{-1}\big)\big|_{\widehat{\Or}'_{I^{-1}}}=(-1)^{n(n-1)/2} A_{J^{-1} \cdot \sigma_{0} ,I^{-1}\cdot \sigma_{0} }\big(z_{\sigma_{0}(1)}',\dots,z_{\sigma_{n}(1)}',1/\mu'_1,\dots, 1/\mu_n'\big).
 \end{gather*}
 From the definition of $\kappa$ we obtain
 \begin{gather*}
 \kappa^{*}\big( {\bf Stab}'\big(J^{-1}\big)\big|_{\widehat{\Or}'_{I^{-1}}}\big)=
 (-1)^{n(n-1)/2} A_{J^{-1}\cdot \sigma_0,I^{-1}\cdot \sigma_{0}}(\mu_{\sigma_{0}(1)},\dots,\mu_{\sigma_{0}(n)}, 1/z_1,\dots, 1/z_n).
 \end{gather*}
Thus, the statement is equivalent to Theorem \ref{duthrm}.
\end{proof}

Our Definition \ref{maindefin} of $3$d mirror symmetry then implies:

\begin{Corollary}The variety $X'$ is a $3$d mirror of $X$.
\end{Corollary}
As $X\cong X'$ we say that $X$ is 3d mirror self-dual.

\section{The duality interface}
\subsection{Interpolation function}
Let us define the following combination of elliptic weight functions
\begin{gather*}
\tilde{\mathfrak{m}}(\ttt,\ttt'):=(-1)^{n(n-1)/2}\sum\limits_{I,J\in \mathfrak{S}_n} A_{I,J}^{-1}(\zz,\zz') W_{J}(\ttt,\zz,\hbar,\zz') W_{I^{-1}\cdot \sigma_{0}}(\ttt',\zz'_{\sigma_0},1/\zz).
\end{gather*}
This function interpolates the elliptic weight functions in the following sense.
\begin{Lemma} \label{motherlemma}
 \begin{gather*}
 \tilde{\mathfrak{m}}(\ttt,\zz'_{I^{-1}})= W_{I}(\ttt,\zz,\hbar,\zz'), \qquad \tilde{\mathfrak{m}}(\zz_{I^{-1}},\ttt')= (-1)^{n(n-1)/2} W_{I \cdot \sigma_{0}}(\ttt',\zz_{\sigma_0}',\hbar,1/\zz).
 \end{gather*}
\end{Lemma}
\begin{proof}
 Obvious from the definition of restriction matrix (\ref{resmat}) and Theorem~\ref{duthrm}.
\end{proof}

Let us consider the scheme $\mathsf{S}(X)\times \mathsf{S}(X') \times \cE_{\bT\times \bT'}$. As before, we assume that the coordinates on
$\mathsf{S}(X)$ are symmetric functions in Chern roots $\ttt$ and coordinates on $\mathsf{S}(X')$
are symmetric functions in $\ttt'$. By definition, $\tilde{\mathfrak{m}}(\ttt,\ttt')$ is symmetric function in $\ttt$ and $\ttt'$. Therefore, it represents a section of certain line bundle on this scheme.
\subsection{Interpolation function as a section of a line bundle}
We would like to rewrite the statement of the previous lemma in geometric terms.
For a fixed point $L \in (X')^{\bT'}$ we denote by $\alpha'_{L}$ the composition of the following maps
\begin{gather*}
\cE_{\bT}\times \cE_{{\rm Pic}(X)}\times \textsf{S}(X) \rightarrow \cE_{\bT'}\times \cE_{{\rm Pic}(X')}\times \textsf{S}(X)\cong \widehat{\Or}'_{L} \times \textsf{S}(X)
\stackrel{e_{L}}{\rightarrow} E_{\bT'}(X')\times \textsf{S}(X)\\
\hphantom{\cE_{\bT}\times \cE_{{\rm Pic}(X)}\times \textsf{S}(X)}{}
 \stackrel{c_{X'}}{\longrightarrow} \textsf{S}(X)\times \textsf{S}(X') \times \cE_{\bT'}\times \cE_{{\rm Pic}(X')}\rightarrow \textsf{S}(X)\times \textsf{S}(X') \times \cE_{\bT\times \bT'},
\end{gather*}
where the first and the last maps are given by $\kappa$ (just a change of variables), $e_{L}$ is the inclusion of the extended orbit $\widehat{\Or}'_{L}$ to the extended cohomology $E_{\bT'}(X')$ (\ref{duext}) and $c_{X'}$ is the elliptic Chern class for~$X$. We denote by
\begin{gather*}
\alpha_{L}\colon \ \cE_{\bT'}\times \cE_{{\rm Pic}(X')}\times \textsf{S}(X')\longrightarrow \textsf{S}(X)\times \textsf{S}(X') \times \cE_{\bT\times \bT'}
\end{gather*}
the map given by the same chain of maps with $X'$ in place of $X$. Lemma~\ref{motherlemma} can be formulated as follows
\begin{Lemma} \label{allem}
 \begin{gather*}
 \alpha^{'*}_{L^{-1}} (\tilde{\mathfrak{m}})= W_{L}(\ttt,\zz,\hbar,\mub), \qquad \alpha^{*}_{L^{-1}} (\tilde{\mathfrak{m}})=(-1)^{n(n-1)/2} W_{L\cdot\sigma_0}(\ttt',\zz'_{\sigma_0},\hbar,1/\mub').
 \end{gather*}
\end{Lemma}
\begin{proof} The map $(c_{X'} \circ e_L)^*$ in $\alpha^{'*}_{L}$ is the restriction of a section to the orbit $\widehat{\Or}'_{L}$. By definition, it is given by a substitution $\ttt'=\zz_{L}'$. The same for $\alpha^{*}_{L}$. The result follows from the Lemma~\ref{motherlemma} after the change of variables by $\kappa$.
\end{proof}

\subsection{The duality interface}
Let us consider a $\bT\times \bT'$-variety $X\times X'$. For fixed points $I\in X^{\bT}$, $J\in (X')^{\bT'}$ we consider the equivariant embeddings
\begin{gather} \label{embeds}
X\times \{ J \} \stackrel{i_{J}}{\longrightarrow} X\times X' \stackrel{i_{I}}{\longleftarrow} \{I\}\times X' .
\end{gather}
We have
\begin{gather*}
{\rm Ell}_{\bT\times\bT'}(X\times \{ J \})={\rm Ell}_{\bT}(X)\times \cE_{\bT'}\cong \mathsf{E}_{\bT}(X),
\end{gather*}
where the last equality is by (\ref{ident}). Similarly, we use (\ref{ident})
to fix the isomorphism ${\rm Ell}_{\bT\times\bT'}(\{I\}\times X')\cong\mathsf{E}_{\bT'}(X')$. By covariance of the equivariant elliptic cohomology functor, the maps (\ref{embeds}) induce the following embeddings
\begin{gather*}
\mathsf{E}_{\bT}(X) \stackrel{i^{*}_{J}}{\longrightarrow} {\rm Ell}_{\bT\times\bT' }(X\times X' )\stackrel{i^{*}_{I}}{\longleftarrow} \mathsf{E}_{\bT'}(X').
\end{gather*}

\begin{Theorem} There exists a holomorphic section ${\mathfrak{m}}$ $($the \emph{duality interface}\footnote{In the previous paper~\cite{RVSZ}, it is called the \emph{Mother function}.}$)$ of a certain line bundle on ${\rm Ell}_{\bT\times\bT' }(X\times X' )$ such that
 \begin{gather*}
 (i^{*}_{J})^*({\mathfrak{m}})= {\bf Stab}(I), \qquad (i^{*}_{I})^*({\mathfrak{m}})= {\bf Stab}' (J),
 \end{gather*}
 where $I$ is a fixed point on $X$ and $J$ is the corresponding fixed point on $X'$ $($i.e., $J=I^{-1}$ as a~permutation$)$.
\end{Theorem}

\begin{proof} Let
 \begin{gather*}
 {\rm Ell}_{\bT\times\bT'}(X\times X' )\stackrel{c}{\longrightarrow} \textsf{S}(X)\times \textsf{S}(X') \times \cE_{\bT\times \bT'}
 \end{gather*}
 be the embedding by the elliptic Chern classes. Define ${\mathfrak{m}}=c^{*}(\tilde{{\mathfrak{m}}})$.
 For $I\in X^{\bT}$ we can factor the inclusion map as $i^{*}_{I}=\alpha_{I}\circ c_{X'}$ where
 $ c_{X'}\colon E_{\bT}(X') \rightarrow \mathsf{S}(X')\times \cE_{\bT'}\times \cE_{{\rm Pic}(X')} $
 the elliptic Chern classes of tautological bundles over $X'$. Thus,
 \begin{gather*}
 (i^{*}_{I})^{*}(\mathfrak{m})=c_{X'}^{*} \circ \alpha^{*}_{I} (\tilde{\mathfrak{m}}) = c_{X'}^{*} (W_{I^{-1}\cdot \sigma_0}(\ttt',\zz,\hbar,1/\mub'))={\bf Stab}'(I^{-1})={\bf Stab}'(J),
 \end{gather*}
 where the second equality is by Lemma \ref{allem} and the third is by (\ref{duenv}). The calculation for a fixed point on $J \in (X')^{\bT'}$ is the same.

 Finally, by definition, ${\mathfrak{m}}$ is holomorphic if every restriction $ {\mathfrak{m}} |_{\Or_{I,J}}$ is holomorphic. But
 \begin{gather*}
 {\mathfrak{m}} |_{\Or_{I,J}}=A_{I,J}(\zz,\zz'),
 \end{gather*}
 which is holomorphic by Lemma \ref{hollem}.
\end{proof}

\subsection*{Acknowledgments}

The authors would like to thank M.~Aganagic and A.~Okounkov for their insights on 3d mirror symmetries and elliptic stable envelopes that motivates this work. We thank I.~Cherednik for his interest in this work and useful comments. R.R.~is supported by the Simons Foundation grant 523882. A.S.~is supported by RFBR grant 18-01-00926 and by AMS travel grant. A.V.~is supported in part by NSF grant DMS-1665239. Z.Z.~is supported by FRG grant 1564500.

\pdfbookmark[1]{References}{ref}
\LastPageEnding

\end{document}